\documentclass{amsart}
\usepackage{amsfonts, amsmath, amssymb, amscd, amsthm}

\theoremstyle{plain}
\newtheorem{theo}{Theorem}[section]
\newtheorem*{theorem*}{Main Theorem}
\newtheorem{lemm}{Lemma}[section]

\newtheorem{prop}{Proposition}[section]

\theoremstyle{definition}

\newtheorem{rema}{Remark}[section]

\newcommand{\DP}[1]{(\tilde{\nabla}_{#1}P)}
\newcommand{\DPJ}[1]{(\tilde{\nabla}_{#1}(PJ))}
\newcommand{\nks}{\ensuremath{\mathbb{S}^3 \times
\mathbb{S}^3}}   
\newcommand{\e}{\varepsilon}

\numberwithin{equation}{section}

\normalsize
\begin{document}

\title[Isotropic Lagrangian submanifolds
in NK $\mathbb{S}^3\times \mathbb{S}^3$]
{On isotropic Lagrangian
submanifolds in the homogeneous nearly K\"ahler $\mathbb{S}^3\times
\mathbb{S}^3$}

\author{Zejun Hu and Yinshan Zhang}

\thanks{This project was supported by NSFC (Grant No. 11371330)}

\begin{abstract}
In this paper, we show that isotropic Lagrangian submanifolds in a
$6$-dimensional strict nearly K\"ahler manifold are totally
geodesic. Moreover, under some weaker conditions, a complete
classification of the $J$-isotropic Lagrangian submanifolds in the
homogeneous nearly K\"ahler $\mathbb{S}^3\times \mathbb{S}^3$ is
also obtained. Here, a Lagrangian submanifold is called
$J$-isotropic, if there exists a function $\lambda$, such that
$g((\nabla h)(v,v,v),Jv)=\lambda$ holds for all unit tangent vector
$v$.

\end{abstract}

\maketitle

\footnote{2010 \textit{Mathematics Subject Classification}. 53B35,
53C30, 53C42, 53D12.}

\footnote{\textit{Key words and phrases}. Nearly K\"ahler
$\mathbb{S}^3 \times \mathbb{S}^3$, Lagrangian submanifold,
isotropic submanifold, $J$-parallel, totally geodesic.}

\section{Introduction}

Nearly K\"ahler (abbrev. NK) manifolds are a special class of almost
Hermitian manifolds with almost complex structure $J$ satisfying
that $\tilde \nabla J$ is skew-symmetric. In 1970s, A. Gray made
systemically studies on the geometry of NK manifolds (cf. \cite{Gr1,
Gr3}). Towards the important problem of classifying the NK
manifolds, P. A. Nagy made significant contributions on the
decomposition of complete, simply connected, strict NK manifolds,
and in his works \cite{Nag1, Nag2} it was shown that $6$-dimensional
NK manifolds play an distinguished role for the study of generic NK
manifolds. In \cite{But1, But2}, J. B. Butruille further proved that
homogeneous $6$-dimensional NK manifolds must be the NK
$\mathbb{S}^6$, $\mathbb{S}^3 \times \mathbb{S}^3$, the complex
projective space $\mathbb{C}P^3$ and the flag manifold
$SU(3)/U(1)\times U(1)$. On the other hand, we would mention the
recent remarkable development that L.~Foscolo and M.~Haskins
\cite{F-H} have constructed inhomogeneous NK structures on both
manifolds of $\mathbb{S}^6$ and $\mathbb{S}^3 \times \mathbb{S}^3$.

Amongst the geometry of submanifolds of the NK manifolds, most
researches concentrate on the homogeneous NK $\mathbb{S}^6$, and
there are rich literatures (cf. \cite{B-V-W, D-O-V-V, D-V-V2, Ej,
Lo, Vr} and the references therein). We also noticed that recently a
broader study of submanifolds in NK manifolds was investigated in
\cite{S-S} by Sch\"afer and Smoczyk.

In this paper we mainly restrict to study submanifolds of the
homogeneous NK $\mathbb{S}^3 \times \mathbb{S}^3$. Recall that a
submanifold of an almost Hermitian manifold is called almost
complex, if the almost complex structure $J$ preserves the tangent
space. The study of almost complex surfaces of the homogeneous NK
$\mathbb{S}^3 \times \mathbb{S}^3$ was initiated in \cite{B-D-D-V},
where all almost complex surfaces with parallel second fundamental
form were classified. In \cite{D-L-M-V, D-V1, H-Z}, we can find
further developments on the study of the almost complex surfaces in
NK $\mathbb{S}^3 \times \mathbb{S}^3$.

Recall also that a submanifold $M$ of an almost Hermitian manifold
$N$ is called Lagrangian, if the almost complex structure $J$
interchanges the tangent and the normal spaces and thus the
dimension of $M$ is half the dimension of $N$. After the fruitful
study on the Lagrangian submanifolds of the NK $\mathbb{S}^6$ during
the past several decades, in recent years the study of the
Lagrangian submanifolds of the homogeneous NK $\mathbb{S}^3\times
\mathbb{S}^3$ attracts much attention. It is well known that both
factors, $\mathbb{S}^3\times \{pt\}$ and $\{pt\}\times
\mathbb{S}^3$, and the diagonal $\{(x,x)\,|\,x\in\mathbb{S}^3\}$ are
examples of totally geodesic Lagrangian submanifolds in the NK
$\mathbb{S}^3\times \mathbb{S}^3$, in addition to these simplest
ones, A.~Moroianu and U.~Semmelmann \cite{M-S1} constructed several
interesting and new examples of Lagrangian submanifolds of the NK
$\mathbb{S}^3 \times \mathbb{S}^3$. In fact, as usual denoting
$\mathbb{H}$ the quaternion space with $\mathbf{i},\mathbf{j},
\mathbf{k}$ its imaginary units and regarding the unit $3$-sphere
$\mathbb{S}^3$ as the set of all the unit quaternions in
$\mathbb{H}$, it was shown in \cite{M-S1} that the graphs of the
product of some simple functions on $\mathbb{S}^3$ yield Lagrangian
submanifolds of the NK $\mathbb{S}^3 \times \mathbb{S}^3$ with
significantly properties. Furthermore, very recent investigations
reveal that actually all the totally geodesic Lagrangian
submanifolds can be produced in such natural way (see
\cite{Z-D-H-V-W}), and so do the Lagrangian submanifolds with
nonzero constant sectional curvature (see \cite{D-V-W}). In this
respect, we would introduce the following interesting result:
\begin{theo}[\cite{Z-D-H-V-W}]\label{thm:1.1}
A Lagrangian submanifold of any $6$-dimensional strict NK manifold
is of parallel second fundamental form if and only if it is totally
geodesic.
\end{theo}
In fact, as the main result of \cite{Z-D-H-V-W}, after finishing the
proof of Theorem \ref{thm:1.1}, a complete classification of all
totally geodesic Lagrangian submanifolds of the homogeneous NK
$\mathbb{S}^3 \times \mathbb{S}^3$ was explicitly demonstrated,
which consists of six non-congruent examples.

Still focusing on the study of Lagrangian submanifolds of the NK
manifolds, in this paper we first study its subclass of the
so-called {\it isotropic} ones. According to B. O'Neill \cite{ON}, a
submanifold of a Riemannian manifold is called isotropic if and only
if for any tangent vector $v$ at a point $p$, we have the relation
that
\begin{equation}\label{eqn:1.0}
g(h(v,v),h(v,v))=\mu^2(p) (g(v,v))^2,
\end{equation}
where $h$ denotes the second fundamental form of the immersion and
$\mu$ is a non-negative function on the submanifold. Similarly as
for Lagrangian submanifolds of K\"ahler manifolds, studied by
Montiel and Urbano \cite{M-U}, in present setting the following
result holds, which gives a remarkable counterpart of Theorem
\ref{thm:1.1}.
\begin{theo}\label{thm:1.2}
A Lagrangian submanifold of any $6$-dimensional strict NK manifold
is isotropic if and only if it is totally geodesic.
\end{theo}

We remark that Theorem \ref{thm:1.2}, together with Theorem
\ref{thm:1.1}, implies that for Lagrangian submanifolds of any
$6$-dimensional strict NK manifold the isotropic condition is
equivalent to the condition that the second fundamental form is
parallel. As generally a Lagrangian submanifold of arbitrary
$6$-dimensional strict NK manifold is not obviously {\it
curvature-invariant} (see \cite{M-U} for the notion and
\eqref{eqn:2.10} for the assertion), Theorem \ref{thm:1.2} gives a
meaningful generalization of Proposition 1 in \cite{M-U} which
states that if a Lagrangian submanifold in a K\"ahler manifold is
curvature-invariant, then the constant isotropic condition implies
that the second fundamental form is parallel.

Next to isotropic submanifolds which correspond to satisfying the
restriction \eqref{eqn:1.0}, for Lagrangian submanifolds of a NK
manifold, one can further consider the subclass of Lagrangian
submanifolds which satisfy the condition that for any tangent vector
$v$ at a point $p$, we have the {\it special isotropic} relation
\begin{equation}\label{eqn:1.1}
g((\nabla h)(v,v,v),Jv)=\lambda(p) (g(v,v))^2
\end{equation}
for a function $\lambda$ on the submanifold.

For simplicity, a Lagrangian submanifold satisfying \eqref{eqn:1.1}
will be called as {\it $J$-isotropic}. In particular, following the
terminology of \cite{D-V} (where the authors only considered the NK
$\mathbb{S}^6$), we will call a $J$-isotropic Lagrangian submanifold
with vanishing $\lambda$ being {\it $J$-parallel}.

Recall that Djori\'c and Vrancken \cite{D-V} proved that for
Lagrangian manifolds of the homogeneous NK $\mathbb{S}^6$, the
notions of {\it $J$-isotropic} and {\it $J$-parallel} are in fact
equivalent (Theorem B in \cite{D-V}), whereas the $J$-parallel
condition gives more examples than that of parallel second
fundamental from. In fact, by the classification Theorem A in
\cite{D-V}, besides the totally geodesic $3$-sphere, there are two
other examples of $J$-parallel Lagrangian submanifolds in the
homogeneous NK $\mathbb{S}^6$.

The main purpose of the present paper is to extend the results of
\cite{D-V} by replacing the ambient homogeneous NK $\mathbb{S}^6$
with the homogeneous NK $\mathbb{S}^3\times \mathbb{S}^3$. Although
the NK $\mathbb{S}^6$ is a space form whereas the NK
$\mathbb{S}^3\times \mathbb{S}^3$ is much more complicated with the
particular properties that it is even neither locally symmetric nor
Chern flat (cf. \cite{H-Z}), it turns out that, once overcoming the
difficulty brought by the complicated expression of its Riemannian
curvature tensor, we can still succeed in achieving a complete
classification of the {\it $J$-isotropic} Lagrangian submanifolds in
the homogeneous NK $\mathbb{S}^3\times \mathbb{S}^3$. To speak
accurately, we will generalize the results of \cite{D-V} by showing
that the case for NK $\mathbb{S}^3\times \mathbb{S}^3$ is totally
similar as for NK $\mathbb{S}^6$, i.e., for Lagrangian submanifolds
of the homogeneous NK $\mathbb{S}^3\times \mathbb{S}^3$, the two
conditions of {\it $J$-isotropic} and {\it $J$-parallel} are
equivalent. As the main result, a classification theorem can be
obtained as follows:
\begin{theorem*}\label{thm:000}
Let $M$ be a Lagrangian submanifold in the homogeneous NK $\mathbb{S}^3\times
\mathbb{S}^3$. If $M$ is $J$-isotropic in the sense of \eqref{eqn:1.1},
then the isotropic function $\lambda$ vanishes, and $M$ is locally
given by one of the following immersions:
\begin{enumerate}
\item[(1)] \qquad
$f_1:\ \ \mathbb{S}^3 \to\nks\ \ defined\ by\ \ u\mapsto (1,u)$.

\vskip 1mm

\item[(2)]\qquad
$f_2:\ \ \mathbb{S}^3 \to\nks\ \ defined\ by\ \ u\mapsto (u,1)$.

\vskip 1mm

\item[(3)]\qquad
$f_3:\ \ \mathbb{S}^3\to\nks\ \ defined\ by\ \ u\mapsto (u,u)$.

\vskip 1mm

\item[(4)]\qquad
$f_4:\ \ \mathbb{S}^3\to\nks\ \ defined\ by\ \
      u\mapsto (u,u\mathbf{i})$.

\vskip 1mm

\item[(5)]\qquad
$f_5:\ \ \mathbb{S}^3\to \nks\ \ defined\ by\ \
 u\mapsto(u\mathbf{i}u^{-1},u^{-1})$.

\vskip 1mm

\item[(6)]\qquad
$f_6:\ \ \mathbb{S}^3\to \nks\ \ defined\ by\ \
     u\mapsto(u^{-1},u\mathbf{i}u^{-1})$.

\vskip 1mm

\item[(7)]\qquad
$f_7:\ \ \mathbb{S}^3 \to\nks\ \ defined\ by\ \ u\mapsto
(u\mathbf{i}u^{-1},u\mathbf{j}u^{-1})$.

\vskip 1mm

\item[(8)]\qquad
$f_8:\ \ \mathbb{R}^3 \to\nks\ \ defined\ by\ \
(u,v,w)\mapsto(p(u,w),q(u,v))$, where $p$ and $q$ are constant mean
curvature tori in $\mathbb{S}^3$, given respectively by
$$
\begin{aligned}
p(u,w)&=\Big( \cos(\tfrac{\sqrt{3}u}{2}) \cos(\tfrac{\sqrt{3}w}{2}),
        \cos(\tfrac{\sqrt{3}u}{2})\sin(\tfrac{\sqrt{3}w}{2}),\\
        &\hspace{49mm} \sin(\tfrac{\sqrt{3}u}{2}) \cos(\tfrac{\sqrt{3}w}{2},
        \sin(\tfrac{\sqrt{3}u}{2} )\sin(\tfrac{\sqrt{3}w}{2}) \Big),
\end{aligned}
$$
$$
\begin{aligned}
q(u,v)&=\tfrac{1}{\sqrt{2}}\Big(\cos(\tfrac{\sqrt{3}v}{2})(\sin(\tfrac{\sqrt{3}u}{2})+\cos(\tfrac{\sqrt{3}u}{2})),
 \sin(\tfrac{\sqrt{3}v}{2})(\sin(\tfrac{\sqrt{3}u}{2})+\cos(\tfrac{\sqrt{3}u}{2})),\\
 &\ \ \ \ \ \ \ \cos(\tfrac{\sqrt{3}v}{2})(\sin(\tfrac{\sqrt{3}u}{2})-\cos(\tfrac{\sqrt{3}u}{2})),
 \sin(\tfrac{\sqrt{3}v}{2})(\sin(\tfrac{\sqrt{3}u}{2})-\cos(\tfrac{\sqrt{3}u}{2}))
  \Big).
 \end{aligned}
$$
\end{enumerate}
\end{theorem*}
\begin{rema}
The immersions $\{f_i\}_{i=1}^6$ are totally geodesic, and the
immersions $f_1,f_2,f_3$, $f_7, f_8$ are of constant sectional
curvature. It should be pointed out that due to the restriction of
the almost product $P$ of the NK $\mathbb{S}^3\times \mathbb{S}^3$,
the two pairs of $f_1$ and $f_2$, as well as $f_5$ and $f_6$, are
not congruent (cf. \cite{Z-D-H-V-W}). Therefore, combining the main
results of \cite{D-V-W} and \cite{Z-D-H-V-W}, the above Main Theorem
shows that the $J$-isotropic condition \eqref{eqn:1.1} is a very
nice concept which gives a unified characterization of all the
totally geodesic Lagrangian submanifolds together with those of
constant sectional curvature.
\end{rema}
\begin{rema}
Related to our result, it is worth mentioning that isotropic
Lagrangian submanifolds of the complex space forms have been
completely classified (cf. \cite{L-W, M-U, Vr1}). Also, Wang, Li and
Vrancken \cite{W-L-V} have considered Lagrangian submanifolds of the
complex space forms with an {\it isotropic cubic form}, and a
complete classification is obtained if the submanifolds are of
dimension $3$.
\end{rema}

\vskip 2mm

The paper is organized as follows. In section 2, we review relevant
materials for NK manifolds and their Lagrangian submanifolds,
particularly concern is for the homogeneous NK $\mathbb{S}^3\times
\mathbb{S}^3$. In section 3, a proof of Theorem \ref{prop:1.1} is
given. In section 4, restricting to the homogeneous NK
$\mathbb{S}^3\times \mathbb{S}^3$, we first derive an equivalent
statement of the $J$-isotropic condition, then using Ricci identity
we prove Proposition \ref{prop:4.2} which becomes crucial for our
purpose. In section 5, we discuss the examples $f_7$ and $f_8$ as
described in the Main Theorem and show that both of them are
$J$-parallel. Finally in section 6, we complete the proof of the
Main Theorem.

\vskip 2mm

\noindent{\bf Acknowledgements}. We are greatly indebted to Prof.
Luc Vrancken for his very helpful suggestions as well as his
valuable comments during the time when we were working on this
project.


\section{Preliminaries}

\subsection{\bf The homogeneous nearly K\"ahler $\mathbb{S}^3\times \mathbb{S}^3$}

In this subsection, we recall from \cite{B-D-D-V} the homogeneous NK
structure on $\mathbb{S}^3\times \mathbb{S}^3$. By the natural
identification $T_{(p,q)}(\mathbb{S}^3 \times \mathbb{S}^3)\cong
T_p\mathbb{S}^3 \oplus T_q\mathbb{S}^3$, we can write a tangent
vector at $(p,q)$ as $Z(p,q)=(U_{(p,q)},V_{(p,q)})$, or $Z=(U,V)$
for simplicity. The well known almost complex structure $J$ on
$\mathbb{S}^3\times \mathbb{S}^3$ is given by
\begin{equation}\label{eqn:2.1}
J Z(p,q) = \tfrac{1}{\sqrt{3}} (2pq^{-1}V - U, -2qp^{-1}U + V).
\end{equation}

Define the following Hermitian metric $g$
\begin{equation}\label{eqn:2.2}
\begin{split}
g(Z, Z')& = \tfrac{1}{2}(\langle Z,Z'\rangle  + \langle JZ, JZ' \rangle )\\
& = \tfrac{4}{3}(\langle U,U' \rangle + \langle V,V' \rangle)
   - \tfrac{2}{3}(\langle p^{-1}U, q^{-1}V' \rangle + \langle p^{-1}U', q^{-1}V
   \rangle),
\end{split}
\end{equation}
where $Z=(U, V)$, $ Z'=(U', V')$ are tangent vectors and
$\langle\cdot,\cdot\rangle$ is the standard product metric on
$\mathbb{S}^3 \times \mathbb{S}^3$. Then $(\mathbb{S}^3 \times
\mathbb{S}^3,g,J)$ is a homogeneous NK manifold, i.e.,
$(\tilde\nabla_ZJ)Z=0$ holds for any $X\in T_{(p,q)}(\mathbb{S}^3
\times \mathbb{S}^3)$, where $\tilde\nabla$ is the Levi-Civita
connection of $g$.

As usual we denote the tensor field $\tilde \nabla J$ by $G$.
Straightforward computations show that $G$ have the following
properties (cf. \cite{B-D-D-V}):
\begin{gather}
G(X,Y)+G(Y,X)=0,\label{eqn:2.3}\\
G(X,JY)+JG(Y,X)=0,\label{eqn:2.4}\\
g(G(X,Y),Z)+g(G(X,Z),Y)=0,\label{eqn:2.5}\\
(\tilde\nabla_X
G)(Y,Z)=\tfrac{1}{3}(g(Y,JZ)X+g(X,Z)JY-g(X,Y)JZ.\label{eqn:2.6}
\end{gather}

On $\mathbb{S}^3 \times \mathbb{S}^3$, an almost product structure
$P$ can be defined by (cf. \cite{B-D-D-V})
\begin{equation}\label{eqn:2.7}
PZ = (pq^{-1}V, qp^{-1}U),\ \
 \forall~ Z = (U,V)
\in T_{(p,q)}(\mathbb{S}^3 \times \mathbb{S}^3).
\end{equation}

It is easily seen that the operator $P$ is compatible and symmetric
with respect to $g$, and it is anti-commutative with $J$. We
particularly mention that $P$ plays an important role in the study
of submanifolds of the homogeneous NK $\mathbb{S}^3 \times
\mathbb{S}^3$. The following formula allows to express $\tilde
\nabla P$ in terms of $P$ and $G$ (cf. \cite{B-D-D-V}):
\begin{equation}\label{eqn:2.8}
JG(X,PY)+JPG(X,Y)=2\DP{X}Y.
\end{equation}

By definition, the above useful relation immediately yields
\begin{equation}\label{eqn:2.9}
-G(X,PY)+PG(X,Y)=2\DPJ{X}Y.
\end{equation}

The curvature tensor $\tilde R$ of the homogeneous NK $\mathbb{S}^3 \times
\mathbb{S}^3$ is given by (cf. \cite{B-D-D-V}):
\begin{equation}\label{eqn:2.10}
\begin{split}
\tilde{R}(X,Y)Z =& \tfrac{5}{12}\big(g(Y,Z)X - g(X,Z)Y\big)\\
 &+ \tfrac{1}{12}\big(g(JY,Z)JX - g(JX,Z)JY - 2g(JX,Y)JZ\big)\\
 &+ \tfrac{1}{3}\Big(g(PY,Z)PX - g(PX,Z)PY\\
 &\qquad + g(JPY,Z)JPX - g(JPX,Z)JPY\Big).
\end{split}
\end{equation}


\subsection{\bf Lagrangian submanifolds of the NK $\mathbb{S}^3 \times \mathbb{S}^3$}

Let $M$ be a Lagrangian submanifold of the homogeneous NK $\mathbb{S}^3 \times \mathbb{S}^3$,
and therefore $JTM=T^\bot M$. By a result of Sch\"afer and Smoczyk in \cite{S-S}
(see also \cite{G-I-P}), $M$ is orientable and minimal, and that $G(X,Y)$ is
a normal vector for any $X,Y \in TM$. We denote by $\nabla$ and $\nabla^\bot$
respectively the induced connection and normal connection on $M$. Then the
formulas of Gauss and Weingarten can be expressed as
\begin{equation}\label{eqn:2.11}
\begin{split}
\tilde \nabla_X Y&=\nabla_X Y + h(X,Y),\ \ \forall X,Y\in TM,\\
\tilde \nabla_X \xi&=- A_\xi X + \nabla^\bot _X \xi,\ \ \forall X\in
TM,\ \xi\in T^\perp M,
\end{split}
\end{equation}
where $h$ is the second fundamental form, and it is related to the
shape operator $A_\xi$ by
$
g(h(X,Y),\xi)=g(A_\xi X,Y).
$
From \eqref{eqn:2.11} and the property of $G$, we derive
\begin{equation}\label{eqn:2.12}
\nabla^\bot_X JY=G(X,Y) + J\nabla_X Y, \ \ A_{JX}Y=-J h(X,Y).
\end{equation}
Clearly, the second formula in \eqref{eqn:2.12} shows that
$g(h(X,Y),JZ)$ is totally symmetric.

Denote the curvature tensor of $\nabla$ and $\nabla^\bot$ by $R$ and
$R^\bot$, respectively. Then the equations of Gauss, Codazzi and
Ricci are given by
\begin{gather}
R(X,Y,Z,W)=\tilde R(X,Y,Z,W)+ g(h(X,W),h(Y,Z))-g(h(X,Z),h(Y,W)),\label{2.13}\\
g((\nabla h)(X,Y,Z)-(\nabla h)(Y,X,Z),\xi)=g(\tilde R(X,Y)Z,\xi),\label{2.14}\\
g(R^\bot(X,Y)\xi,\eta)=g(\tilde R(X,Y)\xi,\eta)+g([A_\xi A_\eta]X,Y),\label{2.15}
\end{gather}
where $(\nabla h)(X,Y,Z)=\nabla^\bot_X
h(Y,Z)-h(\nabla_XY,Z)-h(Y,\nabla_X Z)$.

Applying the equations of Gauss and Ricci, and using \eqref{eqn:2.10}
as well as \eqref{eqn:2.12}, we find
\begin{equation}\label{eqn:2.16}
g(R(X,Y)Z,W)=
g(R^\bot(X,Y)JZ,JW)+\tfrac{1}{3}\big(g(X,W)g(Y,Z)-g(X,Z)g(X,W)\big).
\end{equation}

Related to $\nabla h$, we also have the following useful formula
(cf. Lemma 2.3 in \cite{Z-D-H-V-W}):
\begin{equation}\label{eqn:2.17}
g((\nabla h)(X,Y,Z),JW)-g((\nabla h)(X,Y,W),JZ)= g(h(X,Y),G(W,Z)).
\end{equation}

Define the second covariant derivative $\nabla^2 h$ by
\begin{equation*}
\begin{split}
(\nabla^2 h)(X,Y,Z,W)=&\nabla^\bot_X ((\nabla h)(Y,Z,W))-(\nabla h)(\nabla_X Y,Z,W)\\
                      & -(\nabla h)(Y,\nabla_X Z,W)-(\nabla h)(Y, Z,\nabla_X W).
\end{split}
\end{equation*}
Then $\nabla^2 h$ satisfies the Ricci identity
\begin{equation}\label{eqn:2.18}
\begin{split}
(\nabla^2 h)&(X,Y,Z,W)-(\nabla^2 h)(Y,X,Z,W)\\&=
R^\bot(X,Y)h(Z,W)-h(R(X,Y)Z,W)-h(Z,R(X,Y)W).
\end{split}
\end{equation}

Since $M$ is Lagrangian, the pull-back of~$T(\mathbb{S}^3\times
\mathbb{S}^3)$ to~$M$ splits into~ $TM\oplus JTM$. We can consider
on a dense open set of $M$, and choose an orthonormal frame field
$\{e_1,e_2,e_3\}$ such that (cf. \cite{D-V-W}, p.8)
\begin{equation}\label{eqn:2.19}
Pe_1=\lambda_1e_1+\mu_1Je_1,\
Pe_2=\lambda_2e_2+\mu_2Je_2,\
Pe_3=\lambda_3e_3+\mu_3Je_3,
\end{equation}
where $\lambda_i=\cos 2\theta_i,\ \mu_i=\sin 2\theta_i,\ i=1,2,3.$
The functions $\theta_1,\theta_2$ and $\theta_3$ are called the
angles of $M$. Taking into account of the properties of $G$, we can
further assume
$$
\sqrt{3}JG(e_i,e_j)=\sum_k\e_{ij}^ke_k,
$$
%
where $\e_{ij}^k$ is the Levi-Civita symbol, i.e.,
$$
\e_{ij}^k: =\left\{
\begin{aligned}
& 1, \qquad\,  {\rm if}\ (ijk)\  {\rm is\ an\ even\ permutation\ of\ (123)},\\
&-1,\quad  {\rm if}\ (ijk)\ {\rm is\ an\ odd\ permutation\ of\ (123)},\\
& 0, \ \qquad {\rm otherwise}.
\end{aligned}
\right.
$$

Denote by $\omega_{ij}^k$ the coefficient of induced connection,
determined by
$$
\nabla_{E_i}E_j=\sum_k\omega_{ij}^k E_k,\quad
\omega_{ij}^k=-\omega_{ik}^j.
$$

In \cite{D-V-W}, applying the properties of $P$, it was proved that
the angle functions $\{\theta_1,\theta_2,\theta_3\}$ satisfy the
following important relations:
\begin{lemm}[\cite{D-V-W}]\label{lem:1}
Let $M$ be Lagrangian submanifold of the homogeneous NK $\mathbb{S}^3 \times \mathbb{S}^3$.
With respect to the above chosen frame field $\{e_1,e_2,e_3\}$, we have
\begin{enumerate}
 \item $\theta_1 + \theta_2 + \theta_3$ is a multiple of $\pi$,
 \item $ e_i(\theta_j) = -h_{jj}^i$,
 \item $ h_{ij}^k \cos(\theta_j - \theta_k)
 = \Bigl(\frac{\sqrt{3}}{6}\e_{ij}^k - \omega_{ij}^k \Bigr)\sin(\theta_j -
\theta_k),\ \forall\,j\not=k$,
\end{enumerate}
where $h_{ij}^k=g(h(e_i,e_j),Je_k)$.
\end{lemm}
It is worth noting that the angle functions $\{\theta_i\}_{i=1}^3$ play an important
role in the study of Lagrangian submanifolds of $\mathbb{S}^3 \times \mathbb{S}^3$.
In fact, by carefully analyzing the angle functions, a complete
classification of the Lagrangian manifolds in the homogeneous NK $\mathbb{S}^3 \times \mathbb{S}^3$
with constant sectional curvature was obtained in \cite{D-V-W}. Also in \cite{D-V-W}, the
authors established the following interesting characterization of totally geodesic Lagrangian
submanifolds in terms of the angle functions.
\begin{lemm}[Lemma 3.8 of \cite{D-V-W}]\label{lem:2}
If two of the angles $\{\theta_1,\theta_2,\theta_3\}$ are equal
modulo $\pi$, then the Lagrangian submanifold is totally geodesic.
\end{lemm}
We remark that the totally geodesic Lagrangian submanifolds in the
homogeneous NK $\mathbb{S}^3 \times \mathbb{S}^3$ have been
classified in \cite{Z-D-H-V-W}. Actually such a Lagrangian
submanifold is locally given by one of the immersions
$\{f_i\}_{i=1}^6$ described in the Main Theorem.

\section{Proof of Theorem \ref{thm:1.2}}\label{sect:3}
To give a proof of Theorem \ref{thm:1.2}, we are sufficient to
consider the ``only if" part.

Let $M$ be an isotropic Lagrangian submanifold of a $6$-dimensional
strict NK manifold. Suppose on the contrary that the assertion is
not true, then we have a point $p\in M$ such that it is not totally
geodesic. Consider the function $F$ defined on the unit sphere
$U_pM$ in $T_pM$ by $F(v)= g(h(v,v),Jv). $ Noting that the cubic
form $g(h(\cdot,\cdot),J\cdot)$ is totally symmetric (cf.
Proposition 3.2 of \cite{S-S}), we can choose an orthonormal basis
$\{ e_1,e_2.e_3\}$ of $T_pM$ as in Lemma 1 of \cite{M-U}, such that
\begin{equation}\label{eqn:3.1}
h(e_1,e_j)=\mu_jJe_j, \ j=1,2,3,
\end{equation}
where $\mu_1$ is the maximum of the function $F$ on $U_pM$.

As $M$ satisfies the condition
\begin{equation}\label{eqn:3.2}
g(h(v,v),h(v,v))=\mu^2 (g(v,v))^2,\ \forall~ v\in TM,
\end{equation}
we derive from \eqref{eqn:3.1} and \eqref{eqn:3.2} that
$\mu_1=\mu\,\ge0$. From the relation \eqref{eqn:3.2} we also have
(see (3.2) in \cite{M-U}):
\begin{equation}\label{eqn:3.3}
\mu^2-g(h(x,x),h(y,y))-2g((h(x,y),h(x,y))=0,
\end{equation}
for all orthonormal vectors $x,y$. Now in \eqref{eqn:3.3} taking
$x=e_1,y=e_j$ for $j=2$ and $j=3$, respectively, then using
\eqref{eqn:3.1} we derive $\mu^2-\mu\mu_j-2\mu_j^2=0$. Solving this
equation for $\mu_j$, we get $\mu_j=-\mu$, or
$\mu_j=\frac{1}{2}\mu$. From Theorem A of \cite{S-S} we know that
$M$ is minimal, thus we have
$$
0=g(h(e_1,e_1)+h(e_2,e_2)+h(e_3,e_3),Je_1)=\mu+\mu_2+\mu_3,
$$
which, together with the previous relations, yields $\mu=0$. Hence,
by \eqref{eqn:3.2}, we easily derive $h(u,v)=0$ for all tangent
vector $u$ and $v$, i.e., $M$ is totally geodesic at $p$, we get the
desired contradiction. \qed


\section{Implications of the $J$-isotropic condition}

Now we assume that $M$ is a $J$-isotropic Lagrangian submanifold in
the homogeneous NK $\mathbb{S}^3 \times \mathbb{S}^3$, such that
\begin{equation}\label{eqn:4.1}
g((\nabla h)(v,v,v),Jv)=\lambda (g(v,v))^2,\ \forall~ v\in TM,
\end{equation}
where $\lambda$ is a function on $M$.

First of all, for later use we present an equivalent condition of
the $J$-isotropic property \eqref{eqn:4.1}. For notational simplicity,
in sequel we will use the symbol $\mathfrak{S}$ to denote cyclic
sum.
\begin{prop}\label{prop:4.1}
A Lagrangian submanifold $M$ of the homogeneous NK $\mathbb{S}^3 \times \mathbb{S}^3$
is $J$-isotropic if and only if the following equation holds:
\begin{equation}\label{eqn:4.2}
\begin{split}
&12g((\nabla h)(Y,Z,W),JV)+ 3\mathop{\mathfrak{S}}_{YZW}g(h(Y,Z),G(W,V)) \\&
     +2\mathop{\mathfrak{S}}_{ZWV}[g(PY,Z)g(PJW,V)-g(PJY,Z)g(PW,V)]\\&
     -4\lambda\mathop{\mathfrak{S}}_{ZWV}g(Y,Z)g(W,V)
     =0,
\end{split}
\end{equation}
where $Y,Z,W,V$ are any vector fields tangent to $M$.
\end{prop}
\begin{proof}
It is easily seen that the ``if" part is trivial.

Now, we consider the ``only if" part. Taking $v=a_1Y+a_2Z+a_3W+a_4V$
for arbitrary real numbers $\{a_1,a_2,a_3,a_4\}$ in \eqref{eqn:4.1},
and comparing the coefficient of the term $a_1a_2a_3a_4$, we obtain
\begin{equation}\label{eqn:4.3}
\begin{split}
&\mathop{\mathfrak{S}}_{YZW}g((\nabla h)(Y,Z,W),JV)
+\mathop{\mathfrak{S}}_{YZV}g((\nabla h)(Y,Z,V),JW)\\&
+\mathop{\mathfrak{S}}_{YWV}g((\nabla h)(Y,W,V),JZ)
+\mathop{\mathfrak{S}}_{ZWV}g((\nabla h)(Z,W,V),JY)\\&
-4\lambda\mathop{\mathfrak{S}}_{ZWV}g(Y,Z)g(W,V)=0.
\end{split}
\end{equation}

Applying the symmetry of $h$, and the Codazzi equation \eqref{2.14},
such as
\begin{equation}\label{eqn:4.4}
g((\nabla h)(Z,Y,W),JV)=g((\nabla h)(Y,Z,W),JV) + \tilde
R(Z,Y,W,JV),
\end{equation}
to \eqref{eqn:4.3}, we can get
\begin{equation}\label{eqn:4.5}
\begin{split}
0=&\ 3g((\nabla h)(Y,Z,W),JV)+\tilde R(Z,Y,W,JV) + \tilde
R(W,Y,Z,JV)\\&
    +3g((\nabla h)(Y,Z,V),JW)+ \tilde R(Z,Y,V,JW) + \tilde R(V,Y,Z,JW)\\&
    +3g((\nabla h)(Y,W,V),JZ) + \tilde R(W,Y,V,JZ) + \tilde R(V,Y,W,JZ)\\&
    +3g((\nabla h)(Z,W,V),JY) + \tilde R(W,Z,V,JY) + \tilde R(V,Z,W,JY)\\&
    -4\lambda\mathop{\mathfrak{S}}_{ZWV}g(Y,Z)g(W,V).
\end{split}
\end{equation}

By the use of \eqref{eqn:2.17}, we can reduce \eqref{eqn:4.5} to be
\begin{equation}\label{eqn:4.6}
\begin{split}
0=&9g((\nabla h)(Y,Z,W),JV)+3g((\nabla h)(Z,Y,W),JV)\\&
   +3\mathop{\mathfrak{S}}_{YZW}g(h(Y,Z),G(W,V))\\&
    +\tilde R(Z,Y,W,JV) + \tilde R(W,Y,Z,JV)+ \tilde R(Z,Y,V,JW) \\&
   + \tilde R(V,Y,Z,JW) + \tilde R(W,Y,V,JZ) + \tilde R(V,Y,W,JZ)\\&
  + \tilde R(W,Z,V,JY) + \tilde R(V,Z,W,JY)
  -4\lambda\mathop{\mathfrak{S}}_{ZWV}g(Y,Z)g(W,V).
\end{split}
\end{equation}

From \eqref{eqn:4.6}, and applying the Codazzi equation \eqref{eqn:4.4}
once more, we further obtain
\begin{equation}\label{eqn:4.7}
\begin{split}
&12g((\nabla h)(Y,Z,W),JV)+ 3\mathop{\mathfrak{S}}_{YZW}g(h(Y,Z),G(W,V)) \\&
     + 4 \tilde R(Z,Y,W,JV) + \tilde R(W,Y,Z,JV)
     + \tilde R(Z,Y,V,JW) \\&+ \tilde R(V,Y,Z,JW) + \tilde R(W,Y,V,JZ)
   + \tilde R(V,Y,W,JZ)\\& + \tilde R(W,Z,V,JY) + \tilde R(V,Z,W,JY)
   -4\lambda\mathop{\mathfrak{S}}_{ZWV}g(Y,Z)g(W,V)
   =0.
\end{split}
\end{equation}

From \eqref{eqn:4.7} and \eqref{eqn:2.10}, we obtain the expression
\eqref{eqn:4.2} immediately.
\end{proof}

Next, to achieve further implications of the $J$-isotropic condition
\eqref{eqn:4.1}, we differentiate the equation \eqref{eqn:4.2}. Then,
using \eqref{eqn:2.6}, we can get
\begin{equation}\label{eqn:4.8}
\begin{split}
12&\big[g((\nabla^2 h)(X,Y,Z,W),JV)+g((\nabla h)(Y,Z,W),G(X,V))\big]\\&
  +\mathop{\mathfrak{S}}_{YZW}\big[3g((\nabla h)(X,Y,Z),G(W,V))
  +g(h(Y,Z),JW)g(X,V)\\&
  -g(h(Y,Z),JV)g(X,W)\big] +2\mathop{\mathfrak{S}}_{ZWV}\Big[g(PY,h(X,Z))g(PJW,V)\\&
  +g(\DP{X}Y+Ph(X,Y),Z)g(PJW,V)+g(PY,Z)g(PJW,h(X,V))\\&
  +g(PY,Z)g(\DPJ{X}W+PJh(X,W),V)-g(PJY,h(X,Z))g(PW,V)\\&
  -g(\DPJ{X}Y+PJh(X,Y),Z)g(PW,V) -g(PJY,Z)g(PW,h(X,V))\\&
  -g(PJY,Z)g(\DP{X}W+Ph(X,W),V)-2X(\lambda)g(Y,Z)g(W,V)\Big]
  =0,
\end{split}
\end{equation}
where, besides the basic formulas \eqref{eqn:2.11} and \eqref{eqn:2.12},
we have used \eqref{eqn:4.2} for the expressions of such terms
$g((\nabla h)(\nabla_X Y,Z,W),JV), \ldots, g((\nabla h)(
Y,Z,W),J\nabla_XV)$.


From the equation \eqref{eqn:4.8}, we have the following crucial
proposition.
\begin{prop}\label{prop:4.2}
If $M$ is $J$-isotropic in the homogeneous NK $\mathbb{S}^3 \times \mathbb{S}^3$,
then we have
\begin{equation}\label{eqn:4.9}
\begin{split}
12&g(R^\bot(X,Y)h(Z,W)-h(R(X,Y)Z,W)-h(Z,R(X,Y)W),JV)\\&
+9\big[g((\nabla h)(Y,Z,W),G(X,V))-g((\nabla
h)(X_,Z,W),G(Y,V))\big]\\& +3g(h(Y,Z),JW)g(X,V)-3g(h(X,Z),JW)g(Y,V)
\\& +\mathop{\mathfrak{S}}_{\tiny ZW}\Big[
g(h(X,Z),JV)g(Y,W)-g(h(Y,Z),JV)g(X,W) \\&
+g(PY,Z)g(PX,G(W,V))-g(PX,Z)g(PY,G(W,V))\\& +
g(JPY,Z)g(JPX,G(W,V))-g(JPX,Z)g(JPY,G(W,V))\Big]\\&
+2\mathbf{I}(X,Y,Z,W,V) =0,
\end{split}
\end{equation}
where $\mathbf{I}(X,Y,Z,W,V)$ is defined by
\begin{equation}\label{eqn:4.10}
\begin{split}
\mathbf{I}&(X,Y,Z,W,V):=\mathop{\mathfrak{S}}_{ZWV}\Big[g(PY,h(X,Z))g(PJW,V)\\&
  +g(\DP{X}Y+Ph(X,Y),Z)g(PJW,V)+g(PY,Z)g(PJW,h(X,V))\\&
  +g(PY,Z)g(\DPJ{X}W+PJh(X,W),V)
  -g(PX,h(Y,Z))g(PJW,V)\\&
  -g(\DP{Y}X+Ph(X,Y),Z)g(PJW,V)-g(PX,Z)g(PJW,h(Y,V))\\&
  -g(PX,Z)g(\DPJ{Y}W+PJh(Y,W),V)
  -g(PJY,h(X,Z))g(PW,V)\\&
  -g(\DPJ{X}Y+PJh(X,Y),Z)g(PW,V) -g(PJY,Z)g(PW,h(X,V))\\&
  -g(PJY,Z)g(\DP{X}W+Ph(X,W),V)
  +g(PJX,h(Y,Z))g(PW,V)\\&
  +g(\DPJ{Y}X+PJh(X,Y),Z)g(PW,V) +g(PJX,Z)g(PW,h(Y,V))\\&
  +g(PJX,Z)g(\DP{Y}W+Ph(Y,W),V)-2X(\lambda)g(Y,Z)g(W,V)\\&
  +2Y(\lambda)g(X,Z)g(W,V)\Big].
\end{split}
\end{equation}
\end{prop}
\begin{proof}
This is a direct consequence of \eqref{eqn:4.8} and the following
Ricci identity
\begin{equation*}\begin{split}
g((\nabla^2 h)&(X,Y,Z,W),JV)-g((\nabla^2 h)(Y,X,Z,W),JV)\\&
=g(R^\bot(X,Y)h(Z,W)-h(R(X,Y)Z,W)-h(Z,R(X,Y)W),JV).
\end{split}
\end{equation*}

In fact, the totally symmetry of $g(h(\cdot,\cdot),J\cdot)$ implies
that
\begin{equation}\label{eqn:4.11}
\begin{split}
\mathop{\mathfrak{S}}_{YZW}&\big[g(h(Y,Z),JW)g(X,V)-g(h(Y,Z),JV)g(X,W)\big]\\
=&3g(h(Y,Z),JW)g(X,V)-g(h(Y,Z),JV)g(X,W)\\&
-g(h(Z,W),JV)g(X,Y)-g(h(W,Y),JV)g(X,Z),
\end{split}
\end{equation}
from which we obtain the term
\begin{equation*}
\begin{split}
3g&(h(Y,Z),JW)g(X,V)-3g(h(X,Z),JW)g(Y,V) \\&
+\mathop{\mathfrak{S}}_{\tiny ZW}\big[
g(h(X,Z),JV)g(Y,W)-g(h(Y,Z),JV)g(X,W)\big]
\end{split}
\end{equation*}
in \eqref{eqn:4.9} immediately.

Next, the Codazzi equation implies that
$$
g((\nabla h)(X,Y,Z),G(W,V))-g((\nabla h)(Y,X,Z),G(W,V))=g(\tilde
R(X,Y)Z,G(W,V)),
$$
by which, and using \eqref{eqn:2.10}, we get
\begin{equation}\label{eqn:4.12}
\begin{split}
3\mathop{\mathfrak{S}}_{ZW}&\big[g((\nabla h)(X,Y,Z),G(W,V))-g((\nabla h)(Y,X,Z),G(W,V))\big]\\
=&\mathop{\mathfrak{S}}_{ZW}\Big[g(PY,Z)g(PX,G(W,V))-g(PX,Z)g(PY,G(W,V))\\&
+ g(JPY,Z)g(JPX,G(W,V))-g(JPX,Z)g(JPY,G(W,V))\Big].
\end{split}
\end{equation}
Thus the term
\begin{equation}\label{eqn:4.13}
\begin{split}
\mathop{\mathfrak{S}}_{\tiny ZW}\big[
&g(PY,Z)g(PX,G(W,V))-g(PX,Z)g(PY,G(W,V))\\
+&g(JPY,Z)g(JPX,G(W,V))-g(JPX,Z)g(JPY,G(W,V))\big]
\end{split}
\end{equation}
in \eqref{eqn:4.9} is derived.

The remaining terms in \eqref{eqn:4.9} can be easily obtained by
direct calculations.
\end{proof}
%
\section{Examples of $J$-isotropic and non-totally geodesic}

As have been shown in \cite{D-V-W}, the two Lagrangian immersions
$f_7$ and $f_8$ into $\mathbb{S}^3 \times \mathbb{S}^3$ (see Main
Theorem for their definitions) are of constant sectional curvature
$3/16$ and zero, respectively. Moreover, both of them are not
totally geodesic.

In this section, we further show that these two immersions are the
simplest examples next to the totally geodesic. Precisely, we have
the following fact:
\begin{prop}\label{prop:5.1}
The immersions $f_7$ and $f_8$ are both $J$-isotropic with
$\lambda=0$, i.e., they are in fact $J$-parallel.
\end{prop}
\begin{proof}
According to the calculations in \cite{D-V-W}, with respect to the
frame field $\{e_1,e_2,e_3\}$ that is assumed satisfying
\eqref{eqn:2.19}, we have
\begin{equation}\label{eqn:5.1}
h_{12}^3=\tfrac{1}{4}, \ \ h_{ij}^k=0\ {\rm for\ other\ }i,j,k;\quad
\omega_{ij}^k=\tfrac{\sqrt{3}}{4}\e_{ij}^k,
\end{equation}
for the immersion $f_7$; while for the immersion $f_8$ we have
\begin{equation}\label{eqn:5.2}
h_{12}^3=-\tfrac12, \ \ h_{ij}^k=0\ {\rm for\ other\ }i,j,k; \quad
\omega_{ij}^k=0.
\end{equation}

Moreover, the angles of $f_7$ and $f_8$ are both
given by
\begin{equation}\label{eqn:5.3}
(2\theta_1,2\theta_2,2\theta_3)=(0,2\pi/3,4\pi/3).
\end{equation}

As $h_{ij}^k$ and $\omega_{ij}^k$ are constant for both immersions,
we calculate that
\begin{equation}\label{eqn:5.4}
\begin{split}
&g((\nabla_{e_i} h)(e_j,e_k),Je_l)\\
&=g(\nabla^\bot_{e_i} h(e_j,e_k),Je_l)-g( h(\nabla_{e_i}e_j,e_k),Je_l)-g( h(e_j,\nabla_{e_i}e_k),Je_l)\\
&=\sum_m\big[h_{jk}^m
g(\nabla^\bot_{e_i}Je_m,Je_l)-\omega_{ij}^mh_{mk}^l-\omega_{ik}^mh_{mj}^l\big].
\end{split}
\end{equation}

Noting that
$$
\nabla^\bot_{e_i}Je_m=G(e_i,e_m)+J\nabla_{e_i}e_m, \ \
\sqrt{3}JG(e_i,e_j)=\sum_k\e_{ij}^k e_k,
$$
\eqref{eqn:5.4} implies that
\begin{equation}\label{eqn:5.5}
\begin{split}
g((\nabla_{e_i} h)(e_j,e_k),Je_l)=\sum_m\big[h_{jk}^m
(\tfrac{1}{\sqrt{3}}\e_{mi}^l
+\omega_{im}^l)-\omega_{ij}^mh_{mk}^l-\omega_{ik}^mh_{mj}^l\big].
\end{split}
\end{equation}

To complete the proof, we next show that \eqref{eqn:4.2} holds for
$f_7$ and $f_8$ with $\lambda=0$.

In fact, using \eqref{eqn:5.5}, we see that \eqref{eqn:4.2} becomes
equivalent to the following:
\begin{equation}\label{eqn:5.6}
\begin{split}
12&\sum_m\big[h_{jk}^m (\tfrac{1}{\sqrt{3}}\e_{mi}^l
+\omega_{im}^l)-\omega_{ij}^mh_{mk}^l-\omega_{ik}^mh_{mj}^l\big]\\&
+3\big[g(h(e_i,e_j),G(e_k,e_l))+g(h(e_i,e_k),G(e_j,e_l))+g(h(e_j,e_k),G(e_i,e_l))\big]\\&
+2\big[g(Pe_i,e_j)g(PJe_k,e_l)-g(PJe_i,e_j)g(Pe_k,e_l)\\&
+g(Pe_i,e_k)g(PJe_l,e_j)-g(PJe_i,e_k)g(Pe_l,e_j)\\&
+g(Pe_i,e_l)g(PJe_j,e_k)-g(PJe_i,e_l)g(Pe_j,e_k)\big]\\&
-4\lambda\big[g(e_i,e_j)g(e_k,e_l)+g(e_i,e_k)g(e_l,e_j)+g(e_i,e_l)g(e_j,e_k)\big]=0,\
\ \forall ~i,j,k,l.
\end{split}
\end{equation}

Noticing that at least two indices of $\{i,j,k,l\}$ are the same, by
the facts \eqref{eqn:5.1}-\eqref{eqn:5.3}, we easily see that \eqref{eqn:5.6}
holds for arbitrary $\lambda$ when $\{i,j,k,l\}=\{1,2,3\}$, or three
elements of $\{i,j,k,l\}$ are equal.

Therefore, to show that $f_7$ and $f_8$ are $J$-parallel, it is
sufficient to prove that \eqref{eqn:5.6} holds for $\lambda=0$ in the
three cases: $i=j\neq k=l$, $i=k\neq j=l$ or $i=l\neq j=k$.

As the calculations for the above three cases are straightforward
and totally similar, in below we will only taking for example treat
the case $i=j\neq k=l$. In this case, the equation \eqref{eqn:5.6}
reduces to
\begin{equation}\label{eqn:5.7}
\sum_m \big[2\sqrt{3}h_{ki}^m\e_{mi}^k+12h_{ki}^m\omega_{im}^k
+\sqrt{3}h_{ki}^m\e_{ki}^m+\sin2(\theta_k-\theta_i)\big]-2\lambda=0,\
\ \forall ~i\not=k.
\end{equation}

Using the facts \eqref{eqn:5.1} and \eqref{eqn:5.2}, we further see that
\eqref{eqn:5.7} becomes equivalent to
\begin{equation}\label{eqn:5.8}
2\lambda=\tfrac{\sqrt{3}}{2}\e_{im}^k+\sin2(\theta_k-\theta_i), \
m\neq i,k.
\end{equation}

From the fact \eqref{eqn:5.3}, now it is trivial to check that
\eqref{eqn:5.8} does hold for $\lambda=0$.

This completes the proof of Proposition \ref{prop:5.1}.
\end{proof}

\section{Proof of the Main Theorem}

For simplicity consideration, let us denote
$R_{ijkl}=g(R(e_i,e_j)e_k,e_l)$.

First, we take $X=e_2$, $Y=Z=W=e_1$ and $V=e_3$ in \eqref{eqn:4.9}
to obtain
\begin{equation}\label{eqn:6.1}
\begin{split}
0=&12g(R^\bot(e_2,e_1)h(e_1,e_1)-2h(R(e_2,e_1)e_1,e_1),Je_3) + 2g(h(e_2,e_1),Je_3)\\&
  +9g((\nabla h)(e_1,e_1,e_1),G(e_2,e_3))-9g((\nabla h)(e_2,e_1,e_1),G(e_1,e_3))\\&
 +2\big[g(Pe_1,e_1)g(Pe_2,G(e_1,e_3))+ g(JPe_1,e_1)g(JPe_2,G(e_1,e_3))\big]+2\mathbf{I},
\end{split}
\end{equation}
where $\mathbf{I}=\mathbf{I}(e_2,e_1,e_1,e_1,e_3)$ is given by
\begin{equation}\label{eqn:6.2}
\begin{split}
\mathbf{I}=&g(Pe_1,h(e_2,e_3))g(PJe_1,e_1) +
g\big(\DP{e_2}e_1+Ph(e_2,e_1),e_3\big)g(PJe_1,e_1)\\&
+g(Pe_1,e_1)\big[ g(PJe_1,h(e_2,e_3))+ g(PJe_3,h(e_2,e_1)) \big]
+g(Pe_1,e_1)\cdot\\&\big[ g(\DPJ{e_2}e_1+PJh(e_2,e_1),e_3)+
g(\DPJ{e_2}e_3+PJh(e_2,e_3),e_1) \big]\\&
-g(Pe_2,h(e_1,e_3))g(PJe_1,e_1) -
g\big(\DP{e_1}e_2+Ph(e_1,e_2),e_3\big)g(PJe_1,e_1)\\&
-g(PJe_1,h(e_2,e_3))g(Pe_1,e_1)-
g\big(\DPJ{e_2}e_1+PJh(e_2,e_1),e_3\big)g(Pe_1,e_1)\\&
-g(PJe_1,e_1)\big[ g(Pe_1,h(e_2,e_3))+ g(Pe_3,h(e_2,e_1) )\big]\\&
-g(PJe_1,e_1)\big[ g(\DP{e_2}e_1+Ph(e_2,e_1),e_3)+
g(\DP{e_2}e_3+Ph(e_2,e_3),e_1) \big]\\&
+g(PJe_2,h(e_1,e_3))g(Pe_1,e_1)+g\big(\DPJ{e_1}e_2+PJh(e_1,e_2),e_3\big)g(Pe_1,e_1).
\end{split}
\end{equation}

By using \eqref{eqn:2.19} and Proposition \ref{prop:4.1}, we can
calculate that
\begin{equation}\label{eqn:6.3}
\begin{split}
-9g&((\nabla h)(e_2,e_1,e_1),G(e_1,e_3))\\
 =&~\tfrac{9}{2\sqrt{3}}g(h(e_1,e_2),G(e_1,e_2))\\&
  +\tfrac{\sqrt{3}}{2}\big[g(Pe_2,e_2)g(PJe_1,e_1)-g(PJe_2,e_2)g(Pe_1,e_1)\big]\\
 =&-\tfrac{3}{2}h_{12}^3+\tfrac{\sqrt{3}}{2}(\lambda_2\mu_1-\lambda_1\mu_2).
\end{split}
\end{equation}

Putting \eqref{eqn:6.3} into \eqref{eqn:6.1}, and using
\eqref{eqn:2.16} as well as \eqref{eqn:2.19}, we derive
\begin{equation}\label{eqn:6.4}
\begin{split}
12\big[R_{2113}(h_{11}^1&-2h_{33}^1)+R_{2123}h_{11}^2-2R_{2112}h_{12}^3\big]\\&
+\tfrac{\sqrt{3}}{6}(\lambda_1\mu_2-\lambda_2\mu_1)+\tfrac{1}{2}h_{12}^3
+2\mathbf{I} =3\sqrt{3}\lambda.
\end{split}
\end{equation}

Next, we take $X=e_2$, $Y=Z=V=e_1$ and $W=e_3$ in \eqref{eqn:4.9} to
derive
\begin{equation}\label{eqn:6.5}
\begin{split}
0=&12g(R^\bot(e_2,e_1)h(e_1,e_3)-h(R(e_2,e_1)e_1,e_3)-h(e_1,R(e_2,e_1)e_3),Je_1)\\&
+9g((\nabla h)(e_1,e_1,e_3),G(e_2,e_1)) -
3g(h(e_2,e_1),Je_3)+g(h(e_2,e_3),Je_1)\\&
+g(Pe_1,e_1)g(Pe_2,G(e_3,e_1))+ g(JPe_1,e_1)g(JPe_2,G(e_3,e_1))
+2\mathbf{I}',
\end{split}
\end{equation}
where, by definition, we can check that
$\mathbf{I}'=\mathbf{I}(e_2,e_1,e_1,e_3,e_1)=\mathbf{I}(e_2,e_1,e_1,e_1,e_3)=\mathbf{I}$
as defined by \eqref{eqn:6.2}.

By using \eqref{eqn:2.19} and Proposition \ref{prop:4.1}, we can
calculate that
\begin{equation}\label{eqn:6.6}
\begin{split}
9g&((\nabla h)(e_1,e_1,e_3),G(e_2,e_1))\\
=&-\tfrac{9}{2\sqrt{3}}g(h(e_1,e_3),G(e_1,e_3)\\&
-\tfrac{\sqrt{3}}{2}\big[g(Pe_1,e_1)g(PJe_3,e_3)-g(PJe_1,e_1)g(Pe_3,e_3)\big]\\
=&-\tfrac{3}{2}h_{12}^3-\tfrac{\sqrt{3}}{2}(\lambda_1\mu_3-\lambda_3\mu_1).
\end{split}
\end{equation}

Putting \eqref{eqn:6.6} into \eqref{eqn:6.5}, and using
\eqref{eqn:2.16} and \eqref{eqn:2.19}, we also derive
\begin{equation}\label{eqn:6.7}
\begin{split}
12&\big[R_{2113}(h_{11}^1-2h_{33}^1)+R_{2123}h_{11}^2-2R_{2112}h_{12}^3\big]\\&
-\tfrac{\sqrt{3}}{3}(\lambda_1\mu_2-\lambda_2\mu_1)-\tfrac{\sqrt{3}}{2}(\lambda_1\mu_3-\lambda_3\mu_1)
+\tfrac{1}{2}h_{12}^3+2\mathbf{I}
=0.
\end{split}
\end{equation}

Now from \eqref{eqn:6.4} and  \eqref{eqn:6.7} we easily obtain the
following relation:
\begin{equation}\label{eqn:6.8}
\lambda=\tfrac{1}{6}(\lambda_1\mu_2-\lambda_2\mu_1+\lambda_1\mu_3-\lambda_3\mu_1).
\end{equation}

Similarly, by changing the roles of $e_1,e_2,e_3$ played in the
above discussions in circular order, we also have the following
relations:
\begin{gather}
\lambda=\tfrac{1}{6}(\lambda_2\mu_3-\lambda_3\mu_2+\lambda_2\mu_1-\lambda_1\mu_2),\label{eqn:6.9}\\
\lambda=\tfrac{1}{6}(\lambda_3\mu_1-\lambda_1\mu_3+\lambda_3\mu_2-\lambda_2\mu_3).\label{eqn:6.10}
\end{gather}

Now, taking the summation of \eqref{eqn:6.8}, \eqref{eqn:6.9} and
\eqref{eqn:6.10} immediately yields $\lambda=0$, which verifies the
first assertion of the Main Theorem.

To prove the remaining part of the Main Theorem, we need a more
careful calculation for the expression of $\mathbf{I}$ defined by
\eqref{eqn:6.2}. For that purpose, we make use of \eqref{eqn:2.19} to
simplify $\mathbf{I}$ as follows:
\begin{equation}\label{eqn:6.11}
\begin{split}
\mathbf{I}=&
\mu_1^2h_{12}^3+\big[g(\DP{e_2}e_1,e_3)+\mu_3h_{12}^3\big]\mu_1-\lambda_1(\lambda_1+\lambda_3)h_{12}^3\\&
+\lambda_1\big[g(\DPJ{e_2}e_1,e_3)-\lambda_3h_{12}^3+g(\DPJ{e_2}e_3,e_1)-\lambda_1h_{12}^3\big]\\&
-\mu_2\mu_1h_{12}^3-\big[g(\DP{e_1}e_2,e_3)+\mu_3h_{12}^3\big]\mu_1+\lambda_1^2h_{12}^3\\&
-\big[g(\DPJ{e_2}e_1,e_3)-\lambda_3h_{12}^3\big]\lambda_1-\mu_1(\mu_1+\mu_3)h_{12}^3\\&
-\mu_1\big[g(\DP{e_2}e_1,e_3)+\mu_3h_{12}^3+g(\DP{e_2}e_3,e_1)+\mu_1h_{12}^3\big]\\&
-\lambda_2\lambda_1h_{12}^3+\big[g(\DPJ{e_1}e_2,e_3)-\lambda_3h_{12}^3\big]\lambda_1.
\end{split}
\end{equation}

Noting that $\lambda_i^2+\mu_i^2=1,\ i=1,2,3$, the above expression
reduces to that
\begin{equation}\label{eqn:6.12}
\begin{split}
\mathbf{I}=&-\big[1+2(\lambda_1\lambda_3+\mu_1\mu_3)+(\lambda_1\lambda_2+\mu_1\mu_2)\big]h_{12}^3\\&
 - \mu_1\big[g(\DP{e_1}e_2,e_3)+g(\DP{e_2}e_3,e_1)\big]\\&
 +\lambda_1\big[g(\DPJ{e_1}e_2,e_3)+g(\DPJ{e_2}e_3,e_1)\big].
\end{split}
\end{equation}

Moreover, by use of the formulas \eqref{eqn:2.8} and \eqref{eqn:2.9},
we have the calculations:
\begin{equation}\label{eqn:6.13}
\begin{aligned}
&g(\DP{e_1}e_2,e_3)=\tfrac{1}{2\sqrt{3}}(\lambda_2-\lambda_3),
&& g(\DP{e_2}e_3,e_1)=\tfrac{1}{2\sqrt{3}}(\lambda_3-\lambda_1),\\
&g(\DPJ{e_1}e_2,e_3)=\tfrac{1}{2\sqrt{3}}(\mu_2-\mu_3),
&&g(\DPJ{e_2}e_3,e_1)=\tfrac{1}{2\sqrt{3}}(\mu_3-\mu_1).
\end{aligned}
\end{equation}

Substituting \eqref{eqn:6.13} into \eqref{eqn:6.12} yields
\begin{equation}\label{eqn:6.14}
\begin{split}
\mathbf{I}=&-\big[1+2(\lambda_1\lambda_3+\mu_1\mu_3)+(\lambda_1\lambda_2+\mu_1\mu_2)\big]h_{12}^3\\&
-\tfrac{1}{2\sqrt{3}}\mu_1(\lambda_2-\lambda_1)+\tfrac{1}{2\sqrt{3}}\lambda_1(\mu_2-\mu_1)\\
 =&-\big[1+2(\lambda_1\lambda_3+\mu_1\mu_3)+(\lambda_1\lambda_2+\mu_1\mu_2)\big]h_{12}^3
 +\tfrac{1}{2\sqrt{3}}(\lambda_1\mu_2-\lambda_2\mu_1).
\end{split}
\end{equation}

Putting \eqref{eqn:6.14} into \eqref{eqn:6.4}, with the fact that $\lambda=0$,
we eventually obtain
\begin{equation}\label{eqn:6.15}
\begin{split}
12&\big[R_{2113}(h_{11}^1-2h_{33}^1)+R_{2123}h_{11}^2-2R_{2112}h_{12}^3\big]\\&
-2\big[1+2(\lambda_1\lambda_3+\mu_1\mu_3)+(\lambda_1\lambda_2+\mu_1\mu_2)\big]h_{12}^3\\&
+\tfrac{\sqrt{3}}{2}(\lambda_1\mu_2-\lambda_2\mu_1)+\tfrac{1}{2}h_{12}^3
=0.
\end{split}
\end{equation}


\vskip 3mm

\noindent{\bf Completion of the proof of the Main Theorem}.

If $M$ is a totally geodesic Lagrangian submanifold of the
homogeneous NK $\mathbb{S}^3 \times \mathbb{S}^3$, then it is
trivially $J$-isotropic, in that case, according to the
classification theorem of \cite{Z-D-H-V-W}, $M$ should be given by
one of immersions $\{f_i\}_{i=1}^6$.

Next, we assume that $M$ is $J$-isotropic but not totally geodesic.
We are sufficient to prove that $M$ is given by the immersion $f_7$
or $f_8$.

As we have already proved that $\lambda=0$, the relations
$\lambda_i=\cos2\theta_i$ and $\mu_i=\sin2\theta_i$ enable us to
rewrite the equations \eqref{eqn:6.8}, \eqref{eqn:6.9} and
\eqref{eqn:6.10} as below:
\begin{equation}\label{eqn:6.16}
\left\{
\begin{aligned}
&\cos(\theta_3-\theta_2)\sin(\theta_3+\theta_2-2\theta_1)=0,\\
&\cos(\theta_1-\theta_3)\sin(\theta_1+\theta_3-2\theta_2)=0,\\
&\cos(\theta_2-\theta_1)\sin(\theta_2+\theta_1-2\theta_3)=0.
\end{aligned}
\right.
\end{equation}

From \eqref{eqn:6.16}, we are sufficient to
consider the following three cases.

\vskip 1mm

{\bf Case-I.}\ \ {\it For any distinct $i, j, k$, it hold the
relations}
$$
\sin(\theta_i+\theta_j-2\theta_k)=0,\
\cos(\theta_i-\theta_j)\neq 0.
$$

In this case, we have $\theta_i+\theta_j-2\theta_k\equiv0 \mod\pi$,
for any distinct $i, j, k$. Noting that by Lemma \ref{lem:1},
$\theta_1+\theta_2+\theta_3\equiv0\mod\pi$. It follows that all the
angles $\theta_i\equiv0\mod\pi/3$. As $M$ is not totally geodesic,
then by Lemma \ref{lem:2} all the angles are different modulo $\pi$.
So after rearranging the order of the angle functions if necessary,
only the possibility as
$(\theta_1,\theta_2,\theta_3)=(0,\pi/3,2\pi/3)$ can occur, and
therefore we have
$$
\lambda_1=1,\ \lambda_2=\lambda_3=-1/2,\ \ \mu_1=0,\
\mu_2=-\mu_3=\sqrt{3}/2.
$$

Since all the angles functions are constant, from Lemma \ref{lem:1}
we see that $h_{jj}^i=0$ for all $i,j$. Substituting these relations
into \eqref{eqn:6.15}, we get
\begin{equation}\label{eqn:6.17}
\begin{split}
0=&-24 R_{1221}h_{12}^3+\tfrac{1}{2}h_{12}^3+\tfrac{\sqrt{3}}{2}(\lambda_1\mu_2-\lambda_2\mu_1)\\&
 -2h_{12}^3\big[1+2(\lambda_1\lambda_3+\mu_1\mu_3)+(\lambda_1\lambda_2+\mu_1\mu_2)\big]\\
 =&-24R_{1221}h_{12}^3+\tfrac{3}{2}h_{12}^3+\tfrac{3}{4}.
\end{split}
\end{equation}

Using the following Gauss equation
\begin{equation}\label{eqn:6.18}
R_{2112}=\tfrac{5}{12}+\tfrac{1}{3}(\lambda_1\lambda_2+\mu_1\mu_2)-(h_{12}^3)^2
      =\tfrac{1}{4}-(h_{12}^3)^2,
\end{equation}
we can rewrite \eqref{eqn:6.17} as
\begin{equation}\label{eqn:6.19}
32(h_{12}^3)^3-6h_{12}^3+1=0.
\end{equation}

The above equation for $h_{12}^3$ has exactly only two different
solutions, i.e. $h_{12}^3=1/4$ or $h_{12}^3=-1/2$. Then the Gauss
equations give that
\begin{equation}\label{eqn:6.20}
\begin{split}
R_{ijji}=\tfrac{5}{12}+\tfrac{1}{3}(\lambda_i\lambda_j+\mu_i\mu_j)-(h_{12}^3)^2
=\tfrac{1}{4}-(h_{12}^3)^2,\ \ \forall ~i\neq j.
\end{split}
\end{equation}

If $h_{12}^3=1/4$, then
$M$ has constant sectional curvature $3/16$. By Theorem 5.3 of
\cite{D-V-W}, $M$ is locally given by the immersion $f_7$.

If $h_{12}^3=-1/2$, then $M$ is flat. By Theorem 5.4 of
\cite{D-V-W}, $M$ is locally given by the immersion $f_8$.

\vskip 1mm

{\bf Case-II.}\ \ {\it In $\{\cos(\theta_2-\theta_1),
\cos(\theta_3-\theta_2),\cos(\theta_1-\theta_3)\}$ exactly one is
zero, say
$$
\cos(\theta_2-\theta_1)=0,\ \cos(\theta_3-\theta_2)\neq0,\ \cos(\theta_1-\theta_3)\neq0,
$$
and thus
$\sin(\theta_3+\theta_2-2\theta_1)=\sin(\theta_1+\theta_3-2\theta_2)=0$.}

\vskip 1mm

In this case, we have
\begin{equation}\label{eqn:6.21}
\begin{aligned}
\theta_1-\theta_2&\equiv\tfrac\pi2\mod \pi,\\
\theta_3+\theta_2-2\theta_1&\equiv0\mod\pi,\\
\theta_1+\theta_3-2\theta_2&\equiv0\mod\pi.
\end{aligned}
\end{equation}

As by Lemma \ref{lem:1} we also have
$\theta_1+\theta_2+\theta_3\equiv0\mod\pi$, which is clearly a
contradiction to \eqref{eqn:6.21}. Hence {\bf Case-II} does not occur.

\vskip 1mm

{\bf Case-III.}\ \ {\it At least two elements of
$\{\cos(\theta_2-\theta_1),
\cos(\theta_3-\theta_2),\cos(\theta_1-\theta_3)\}$ are zero, say
$\cos(\theta_2-\theta_1)=\cos(\theta_1-\theta_3)=0$.}

\vskip 1mm

In this case, it holds that
$$
\begin{aligned}
\theta_2-\theta_1&\equiv\tfrac\pi2\mod\pi,\\
\theta_1-\theta_3&\equiv\tfrac\pi2\mod\pi.
\end{aligned}
$$
This yields that $\theta_2-\theta_3\equiv0\mod\pi$, and then by
Lemma \ref{lem:2} we know that $M$ is totally geodesic. This is a
contradiction to our assumption. Hence {\bf Case-III} does not occur
either.

\vskip 1mm

In conclusion, we have completed the proof of the Main Theorem. \qed

%

\vskip 3mm



\vskip 5mm

\begin{flushleft}

{\sc School of Mathematics and Statistics, Zhengzhou University,
Zhengzhou 450001, People's Republic of China.}

E-mails:  huzj@zzu.edu.cn; zhangysookk@163.com

\vskip 2mm

{\it Present address of Yinshan Zhang}:

{\sc School of Sciences, Henan University of Engineering, Zhengzhou,
451191, People's Republic of China.}

\end{flushleft}


\normalsize\noindent
\begin{thebibliography}{11}

\bibitem{B-D-D-V}
J.~Bolton, F.~Dillen, B.~Dioos and L.~Vrancken, {\it Almost
complex surfaces in the nearly K\"ahler $\mathbb{S}^3 \times
\mathbb{S}^3$}, T\^ohoku Math. J. (2) {\bf67} (2015), 1--17.

\bibitem{But1}
J.~B.~Butruille, {\it Classification des vari\'et\'es
approximativement k\"ahleriennes homog\`enes}, Ann. Glob. Anal.
Geom. {\bf27} (2005), 201--225.

\bibitem{But2}
J.~B.~Butruille, {\it Homogeneous nearly K\"ahler manifolds}, in:
Handbook of pseudo-Riemannian geometry and supersymmetry, IRMA Lect.
Math. Theor. Phys., 16, Eur. Math. Soc., Z\"urich, 2010. pp.
399--423.

\bibitem{B-V-W}
J.~Bolton, L.~Vrancken and L.~M.~Woodward, {\it On almost complex
curves in the nearly K\"ahler 6-sphere}, Quart. J. Math.
Oxford. Ser. {\bf45} (1994), 407--427.

\bibitem{D-O-V-V}
F.~Dillen, B.~Opozda, L.~Verstraelen and L.~Vrancken, {\it On totally real
 3-dimensional submanifolds of the nearly Kaehler 6-sphere}, Proc.
 Amer. Math. Soc. {\bf99} (1987), 741--749.

\bibitem{D-V-V2}
F.~Dillen, L.~Verstraelen and L.~Vrancken, {\it Classification of totally
 real 3-dimensional submanifolds of $\mathbb{S}^6(1)$ with $K\ge1/16$}, J.
 Math. Soc. Japan. {\bf42} (1990), 565--584.

\bibitem{D-L-M-V}
B.~Dioos, H.~Li, H.~Ma and L.~Vrancken, {\it Flat almost complex
surfaces in $\mathbb{S}^3 \times \mathbb{S}^3$}. Preprint, 2014.

\bibitem{D-V1}
B.~Dioos, J.~Van der Veken~ and L.~Vrancken, {\it Sequences of harmonic maps in the $3$-sphere},
Math. Nachr. {\bf288} (2015), 2001--2015.

\bibitem{D-V-W}
B.~Dioos, L.~Vrancken and X.~Wang, {\it Lagrangian submanifolds in the nearly
 K\"ahler $\mathbb{S}^3 \times\mathbb{S}^3$}, arXiv:1604.05060v1.

\bibitem{D-V}
M.~Djori\'c and L.~Vrancken, {\it On J-parallel totally real
three-dimensional submanifolds of $\mathbb{S}^6(1)$}, J. Geom. Phys.
{\bf60} (2010), 175--181.

\bibitem{Ej}
N.~Ejiri, {\it Totally real submanifolds in a 6-sphere}, Proc.
Amer. Math. Soc. {\bf83} (1981), 759--763.

\bibitem{F-H}
L.~Foscolo and M.~Haskins, {\it New $G_2$-holonomy cones and exotic
nearly K\"ahler structures on $\mathbb{S}^6$ and $\mathbb{S}^3\times
\mathbb{S}^3$}, arXiv:1501.07838v2.

\bibitem{Gr1}
A. Gray, {\it Nearly K\"ahler manifolds}, J. Diff. Geom. {\bf4} (1970), 283-309.

\bibitem{Gr3}
A. Gray, {\it The structure of nearly K\"ahler manifolds}, Math. Ann. {\bf223}
 (1976), 233-248.

\bibitem{G-I-P}
J.~Gutowski, S.~Ivanov and G. Papadopoulos, {\it Deformations of
generalized calibrations and compact non-K\"ahler manifolds with
vanishing first Chern class}, Asian J. Math. {\bf7} (2003), 39--79.

\bibitem{H-Z}
Z.~Hu and Y.~Zhang, {\it Rigidity of the almost complex surfaces in
the nearly K\"ahler $\mathbb{S}^3 \times\mathbb{S}^3$}, J. Geom.
Phys. {\bf 100} (2016), 80--91.

\bibitem{L-W}
H.~Li and X. Wang, {\it Isotropic Lagrangian submanifolds in complex
Euclidean space and complex hyperbolic space}, Results Math. {\bf56}
(2009), 387--403.

\bibitem{Lo}
J.~D.~Lotay, {\it Ruled Lagrangian submanifolds of the 6-sphere},
Trans. Amer. Math. Soc. {\bf363} (2011), 2305--2339.

\bibitem{M-U}
S.~Montiel and F.~Urbano, {\it Isotropic totally real submanifolds},
Math. Z. {\bf199} (1988), 55--60.

\bibitem{M-S1}
A.~Moroianu and U.~Semmelmann, {\it Generalized Killing spinors and
Lagrangian graphs}, Diff. Geom. Appl. {\bf37} (2014), 141--151.

\bibitem{Nag1}
P. A. Nagy, {\it Nearly K\"ahler geometry and Riemannian foliations},
 Asian J. Math. {\bf6} (2002), 481-504.

\bibitem{Nag2}
P. A. Nagy, {\it On nearly-K\"ahler geometry}, Ann. Glob. Anal. Geom. {\bf22}
(2002), 167--178.

\bibitem{ON}
B. O'Neill,  {\it Isotropic and K\"{a}hler immersions}, Canad. J. Math.
17 (1965), 905-915.

\bibitem{S-S}
L.~Sch\"afer and K.~Smoczyk, {\it Decomposition and minimality of
Lagrangian submanifolds in nearly K\"ahler manifolds}, Ann. Glob.
Anal. Geom. {\bf37} (2010), 221--240.

\bibitem{Vr1}
L.~Vrancken, {\it Some remarks on isotropic submanifolds}, Publ.
Inst. Math. (N.S.) {\bf51(65)} (1992), 94--100.

\bibitem{Vr}
L.~Vrancken, {\it Special Lagrangian submanifolds of the nearly
Kaehler $6$-sphere}, Glasg. Math. J. {\bf45} (2003), 415--426.

\bibitem{W-L-V}
X.~Wang, H.~Li and L.~Vrancken, {\it Lagrangian submanifolds in
3-dimensional complex space forms with isotropic cubic tensor},
Bull. Belg. Math. Soc. Simon Stevin {\bf18} (2011), 431--451.

\bibitem{Z-D-H-V-W}
Y.~Zhang, B.~Dioos, Z.~Hu, L.~Vrancken and X.~Wang, {\it Lagrangian
submanifolds in the $6$-dimensional nearly K\"ahler manifolds with
parallel second fundamental form}, J. Geom. Phys. {\bf 108} (2016),
21--37.


\end{thebibliography}
\end{document}